\documentclass[12pt]{article}
\usepackage[T1]{fontenc}
\usepackage[ansinew]{inputenc}
\usepackage{amsmath}
\usepackage{graphicx}
\title{Optimal constants of $L^2$ inequalities for closed nearly umbilical hypersurfaces in  space forms}
\author{Xu Cheng \thanks{the first author is  partially supported by CNPq and Faperj  of Brazil;}  and Areli V\'azquez Ju\'arez \thanks{the second author is  supported by CONACyT of Mexico and Capes of Brazil. }}

\usepackage{amsmath, amssymb}
\usepackage{amsthm}
\makeatletter

\@addtoreset{equation}{section} \makeatother
\newtheorem{thm}{Theorem}[section]
 
\newtheorem{cor}{Corollary}[section]
 
\newtheorem{prop}{Proposition}[section]\newtheorem{Def}{Definition}[section]

\begin{document}
\date{}
\maketitle
 \begin{abstract}
 Let $\Sigma$ be a  smooth closed hypersurface with  non-negative Ricci curvature, isometrically immersed in a space form. It has been proved in \cite{P}, \cite{CZ}, and \cite{C2} that there are  some  $L^2$ inequalities on $\Sigma$    which measure the stability of closed umbilical hypersurfaces or more generally, closed hypersurfaces with traceless Newton transformation of the second fundamental form.  
In this paper,  we prove that the constants in these  inequalities are optimal. 
 \end{abstract}

\baselineskip=14pt
\section {Introduction}

 A hypersurface $\Sigma$ is called totally umbilical if its second fundamental form $A$ is multiple of its metric $g$ at every point, that is, $A=\frac{\text{tr}A}{n}g$ on $\Sigma$.
In differential geometry, a classical  theorem
states that a closed, i.e., compact and without boundary,  totally umbilical surface isometrically immersed in the Euclidean space $\mathbb{R}^3$ must be a round sphere $\mathbb{S}^2$ and in particular, its second fundamental form $A$ is a constant multiple of the metric.  This result is  also true for hypersurfaces in $\mathbb{R}^{n+1}$.

It is interesting to discuss a quantitative version or stability  of  this theorem. De Lellis and M$\ddot{\text{u}}$ller \cite{dLM} obtained an $L^2$ estimate for closed surfaces in $\mathbb{R}^3$. Their result also has applications to foliations of asymptotically flat three manifolds by surfaces of prescribed mean curvature (\cite{M}, \cite{LM}, \cite{LMS}).  Recently, Perez \cite{P} studied the hypersurface case and proved the following theorem for convex hypersurfaces in $\mathbb{R}^{n+1}$:
 \begin{thm} (\cite{P})\label{thm-3}
Let $\Sigma$ be a smooth, closed and connected hypersurface in $\mathbb{R}^{n+1}, n\geq 2$ with induced metric $g$ and non-negative Ricci curvature, then 
\begin{equation}\label{ine-i2}\int_{\Sigma}|A-\frac{\overline{H}}{n}g|^2\leq \frac{n}{n-1}\int_{\Sigma}|A-\frac{H}{n}g|^2,
\end{equation}
and  equivalently 
\begin{equation}\label{ine-i-02}\int_{\Sigma}\left(H-\overline{H}\right)^2\leq \frac{n}{n-1}\int_{\Sigma}|A-\frac{H}{n}g|^2,
\end{equation}
where  $A$ and  $H=\text{tr} A$ denote   the second fundamental form and the mean curvature  of $\Sigma$ respectively, $\overline{H}=\frac1{\text{Vol}_n(\Sigma)}\int_{\Sigma}H$ is the average of $H$ on $\Sigma$.  In particular, the above estimates hold for smooth, closed hypersurfaces which are the boundary of a convex set in $\mathbb{R}^{n+1}$.
\end{thm}
As pointed  out by De Lellis and Topping \cite{dLT},  Perez's theorem holds even for the closed hypersufaces with nonnegative Ricci curvature when the ambient space is Einstein.   Indeed a slight modification of the  proof of Theorem  \ref{thm-3} gives the following result (see its proof in \cite{CZ}).
\begin{thm}  \label{thm-4}Let $(M^{n+1}, \tilde{g})$ be an Einstein manifold, $ n\geq 2$.
Let $\Sigma$ be a smooth, closed and connected hypersurface immersed in $M$ with induced metric $g$. Assume that $(\Sigma,g)$ has non-negative Ricci curvature.  Then 
\begin{equation}\label{ine-i3}\int_{\Sigma}|A-\frac{\overline{H}}{n}g|^2\leq \frac{n}{n-1}\int_{\Sigma}|A-\frac{H}{n}g|^2,
\end{equation}
and equivalently 
\begin{equation}\label{ine-i-03}\int_{\Sigma}\left(H-\overline{H}\right)^2\leq \frac{n}{n-1}\int_{\Sigma}|A-\frac{H}{n}g|^2,
\end{equation}
where $A$, $H$ and $\overline{H}$ are as in Theorem \ref{thm-3}. 
\end{thm}
Later,  Zhou and the first author  (\cite{CZ}) studied the rigidity of the equalities  in  Inequalities (\ref{ine-i2}),  (\ref{ine-i-02}), (\ref{ine-i3}) and (\ref{ine-i-03}). They proved that
\begin{thm}\label{thm-5} (\cite{CZ})  Let $M^{n+1}$ be  the Euclidean space $\mathbb{R}^{n+1}$, or the Euclidean semi-sphere $\mathbb{S}^n_{+}$ or the  Hyperbolic space $\mathbb{H}^{n+1}$, $n\geq 2$.  Let $\Sigma$ be a smooth, connected, oriented and closed  hypersurface immersed in  $M^{n+1}$ with induced metric $g$. Assume that $(\Sigma,g)$ has non-negative Ricci curvature. Then, 
\begin{equation}\label{eq-i-1} \int_{\Sigma}|A-\frac{\overline{H}}{n}g|^2=\frac{n}{n-1}\int_{\Sigma}|A-\frac{H}{n}g|^2,
\end{equation}
and  equivalently
\begin{equation}\label{eq-i-01}\int_{\Sigma}\left(H-\overline{H}\right)^2=\frac{n}{n-1}\int_{\Sigma}|A-\frac{H}{n}g|^2,
\end{equation}
 holds if and only if $\Sigma$ is a totally umbilical hypersurface,  that is, $\Sigma$ is a distance sphere $S^n$ in $M^{n+1}$, where $\overline{H}=\frac1{\text{Vol}_n(\Sigma)}\int_{\Sigma}H$.

\end{thm}
In \cite{CZ}, the authors also studied the general case for hypersurfaces without assumption on convexity (that is, $A\geq 0$, which is equivalent to $\text{Ric}\geq 0$ when $\Sigma$ is a closed hypersurface in $\mathbb{R}^{n+1}$).  
See more details in \cite{CZ}.

Also in \cite{P}, Perez showed the constants in Inequalities (\ref{ine-i2}) and (\ref{ine-i-02})  are sharp. In this paper, we generalize his result and prove that the constants in Inequalities (\ref{ine-i3}) and (\ref{ine-i-03}) are  sharp when the ambient spaces are other space forms. Precisely, we prove that
 \begin{thm}  \label{opt-1} Let $(M_c^{n+1}, \tilde{g})$ be a space form of constant sectional curvature $c$, $ n\geq 2$. Let  $C<\sqrt{\frac{n}{n-1}}$ be a positive constant. Then   there exists a smooth deformation $\Sigma_t$ of the geodesic sphere $S^n$ in $M_c^{n+1}$ so that for each hypersurface $\Sigma_t$, 

\begin{equation}\label{ine-opt1}\int_{\Sigma_t}\left(H-\overline{H}\right)^2> C^2\int_{\Sigma_t}|A-\frac{H}{n}g|^2,
\end{equation}
and 
\begin{equation}\label{ine-opt1-1}\int_{\Sigma_t}|A-\frac{\overline{H}}{n}g| ^2> C^2\int_{\Sigma_t}|A-\frac{H}{n}g|^2.
\end{equation}

\noindent Moreover, $\Sigma_t$ can be chosen arbitrarily close to $S^n$ so that the Ricci curvature $\text{Ric}_{\Sigma_t}$ of $\Sigma_t$ is positive.

\end{thm}
Theorem \ref{opt-1} shows that the constant $\sqrt{\frac{n}{n-1}}$ in Inequalities (\ref{ine-i3}) and (\ref{ine-i-03})  is optimal when the ambient space is a space form.

In this paper, we also deal with higher order mean curvatures and the Newton transformations of the second fundamental form of hypersurfaces (see their definition in Section \ref{sec-notation}. Besides  we refer the interested readers to \cite{R}, \cite{Ro}, \cite{CR}, \cite{Ro2}, \cite{C2} and  the related references therein).

When $(\Sigma, g)$ is a hypersurface isometrically immersed in a space form,
 it can be verified  that if the Newton transformations $P_r$ satisfy $P_r=\frac{\text{tr}P_r}{n}g$ on $\Sigma$, then $\Sigma$ has constant $r$th mean curvature and thus $P_r$ is a constant multiple of the metric $g$ (cf Section \ref{sec-rmean}). Moreover, Ros' work (\cite{Ro1},\cite{Ro2}) implies 
that  the round spheres are the only closed embedded hypersurfaces  in  $\mathbb{R}^{n+1}$ with $P_r=\frac{\text{tr}P_r}{n}g$, $2\leq r\leq n$. Like the case for the totally umbilical theorem, we  may  consider a quantitative version or stability of this result.
Recently, the first author   \cite{C2} showed  that
 \begin{thm} \cite{C2} \label{thm-rm2} Let $(M_c^{n+1}, \tilde{g})$ be a space form of constant sectional curvature $c$, $ n\geq 2$.
Let $\Sigma$ be a smooth connected  closed  hypersurface immersed in $M_c^{n+1}$ with induced metric $g$. Assume that  $(\Sigma,g)$ has nonnegative Ricci curvature,  then  for $2\leq r\leq n$,
\begin{equation}\label{ine-rm03}(n-r)^2\int_{\Sigma}\left(s_r-\overline{s_r}\right)^2\leq n(n-1)\int_{\Sigma}|P_r-\frac{(n-r)s_r}{n}g|^2,
\end{equation}
and equivalently, 
\begin{equation}\label{ine-rm4}\int_{\Sigma}|P_r-\frac{(n-r)\overline{s_r}}{n}g|^2\leq n\int_{\Sigma}|P_r-\frac{(n-r)s_r}{n}g|^2,
\end{equation}
where  $P_r$ and  $s_r=\text{tr}P_r$  denote the Newton transformations of the second fundamental form $A$ of $\Sigma$ and  the trace of $P_r$ respectively,  $\overline{s_r}=\frac{\int_M s_rdv}{\text{Vol}(M)}$ denotes the average of $s_r$ over $\Sigma$.
\end{thm}

In Section \ref{sec-rmean} of this paper, we consider the optimality  of the constants  for Inequalities (\ref{ine-rm03}) and (\ref{ine-rm4}) and prove that 
\begin{thm}  \label{opt-2} Let $(M_c^{n+1}, \tilde{g})$ be a space form of constant sectional curvature $c$, $ n\geq 2$. Let  the natural number $r$ ($2\leq r\leq n-1$) be given.  For any given constants $C_1<\sqrt{\frac{n(n-1)}{(n-r)^2}}$ and $C_2<\sqrt{n}$, there exist  smooth deformations $(\Sigma_1)_t$ and $(\Sigma_2)_t$ of the geodesic sphere $S^n$ in $M_c^{n+1}$ respectively, so that for each $t$,
\begin{equation}\label{ine-opt2}\int_{(\Sigma_1)_t}\left(s_r-\overline{s_r}\right)^2> C_1^2\int_{(\Sigma_1)_t}|P_r-\frac{(n-r)s_r}{n}g|^2,
\end{equation} 
and 
\begin{equation}\label{ine-opt3}\int_{(\Sigma_2)_t}|P_r-\frac{(n-r)\overline{s_r}}{n}g|^2> C^2_2\int_{(\Sigma_2)_t}|P_r-\frac{(n-r)s_r}{n}g|^2,
\end{equation}
\noindent where $P_r, s_r$ and $ \overline{s_r}$ are given as in Theorem \ref{thm-rm2}. 

Moreover, $(\Sigma_1)_t$ and $(\Sigma_2)_t$ can be chosen arbitrarily close to $S^n$ so that the Ricci curvatures $\text{Ric}_{(\Sigma_i)_t}$ of $(\Sigma_i)_t$, $i=1,2$ are positive.

\end{thm}
Observe that by $P_1=s_1I-A$ and $ s_1=H$,  Theorem \ref{opt-1} shows that Theorem \ref{opt-2} also holds for $r=1$.

It is worth of mentioning that there is a parallel phenomenon in the clue of Riemannian geometry. Recall that  the Schur's theorem states that  the scalar curvature of an Einstein manifold of dimension $n \geq 3$ must be constant. One may consider the stability of Schur's theorem. See some work on this topic in \cite{dLT}, \cite{GWX}, \cite{GW1}, \cite{GW2}, \cite{C1}, \cite{C2}.

The rest of this paper is organized as follows. In Section  \ref{sec-notation}, we give some notation and convention.  In particular, we give the definitions of Newton transformation and $r$-th mean curvatures associated to the second fundamental form.  In Section \ref{sec-var}, we prove the  existence of a smooth normal deformation which is needed in the next two sections. In Section \ref{sec-mean}, we give some evolution equations and prove Theorem \ref{opt-1}. In Section \ref{sec-rmean}, we prove Theorem \ref{opt-2}.
\bigskip

\section{Notation and convention}\label{sec-notation}
In order to give the definition of high order mean curvatures and the Newton transformation associated with the second fundamental form of a hypersurface, which was introduced by Reilly \cite{R}  (cf. \cite{Ro}),
we first recall the definitions of   the $r$th elementary symmetric functions and Newton transformations.
 Let $\sigma_r:\mathbb{R}^r\rightarrow \mathbb{R}$ denote the elementary symmetric function in $\mathbb{R}^n$ given by
$$\sigma_r(x_{1},\ldots,x_{n})=\displaystyle\sum_{i_1<\cdots<i_r}x_{i_1}\cdots x_{i_r}, 1\leq r\leq n.$$
Let $V$ be an $n$-dimensional vector space and  $A:  V\rightarrow V$ be a symmetric linear transformation.  Let $e_1,\ldots, e_n$ denote the orthonormal eigenvectors of $A$ and $\eta_1, \ldots, \eta_n$ denote
the eigenvalues satisfying $Ae_i=\eta_ie_i$, $i=1,\ldots,n$    respectively. 
\begin{Def}\label{elementary}Define the $r$th symmetric functions $\sigma_r(A)$, simply  denoted by $\sigma_r$, associated with $A$ by 
\begin{eqnarray}
\sigma_0&=&1,
\\ \sigma_r&=&\sigma_r(\eta_{1},\ldots, \eta_{n}), 1\leq r\leq n.
\end{eqnarray}
\end{Def}
Also define the  Newton
transformations $P_r: V\rightarrow V, $ associated with $A, 0\leq r\leq n,$  as
\begin{eqnarray}
P_0&=&I\quad  (\text{Identidade}),\\
P_r&=&\displaystyle\sum_{j=0}^r(-1)^j\sigma_{r-j}A^j\nonumber
\\
&=&\sigma_rI-\sigma_{r-1}A+...+(-1)^rA^r, \quad r=1,\ldots,n.\label{def-pr}
\end{eqnarray}
By definition, 
$P_r=\sigma_rI-AP_{r-1},  P_n=0.$ 
 It was proved in  \cite{R} that  $P_r$ has the following basic properties:
\begin{eqnarray}
&(i)&\text{tr}(P_r)
=(n-r)\sigma_r;\label{tr-1}
\\&(ii)&\text{tr}(AP_r)
=(r+1)\sigma_{r+1};\label{tr-2}
\\&(iii)& \text{tr} (A^2P_r)=\sigma_1\sigma_{r+1}-(r+2)\sigma_{r+2},\label{tr-3}
\end{eqnarray}
where $\text{tr}$ denotes the trace of the corresponding transformation.

Now assume $(M, \tilde{g})$ is an $(n+1)$-dimensional Riemannian manifold, $n\geq 2$. Suppose 
 $(\Sigma, g)$ is a smooth connected oriented closed hypersurface immersed in  $(M, \tilde{g})$ with induced metric $g$.  In this paper, unless otherwise specified, we denote by a $``\sim''$ all quantities
 on $(M,\tilde{g})$, for instance  by $ \widetilde{\nabla}$  the Levi-Civita connection of $(M,\tilde{g})$. Also we denote for example by $\nabla, Ric, \Delta$, the Levi-Civita connection, the Ricci curvature tensor, the Laplacian  on  $(\Sigma,g)$ respectively.

  Let  $\nu$ denote the outward unit normal to $\Sigma$. The   second fundamental form $A=(A_{ij})$ of $\Sigma$ is defined by 
   $$A: T_p\Sigma\otimes_s T_p\Sigma\rightarrow \mathbb{R},$$ $$A(X,Y)=-\tilde{g}(\widetilde{\nabla}_XY,\nu),$$ where  $X, Y\in T_p\Sigma, p\in \Sigma$.

 The second fundamental form $A$ corresponds a $(1,1)$-tensor on $T_p\Sigma$, which is called   the shape operator of $\Sigma$ and  still denoted by $A$. Hence the shape operator $A$ satisfies that   $A$:  $T_p\Sigma \rightarrow T_p\Sigma$, $AX=\widetilde{\nabla}_X\nu, X\in T_p\Sigma$.

Let  $\eta_i, i=1,\ldots, n$ denote the principle curvatures  of $\Sigma$ at $p$, which are the eigenvalues of $A$ at $p$ corresponding the orthonormal eigenvectors $\{e_i\}, i=1.\ldots, n$ respectively.  

By Definition \ref{elementary}, 
 for the second fundamental form or the shape operator $A$,  we have $s_r=\sigma_r(A)$, the Newton
transformations $P_r$ associated with $A$ at $p$ respectively, $0\leq r\leq n$ and    
 \begin{Def}
 The $r$th mean
curvature $H_r$ of $\Sigma$ at $p$ is defined by
$s_r=\binom{n}{r}  H_r,$ $ 0\leq r\leq n.$
\end{Def}
For instance, $H_1=\frac{s_1}{n}=\frac{H}{n}$, where $H=\text{tr}A$ is the mean curvature of $\Sigma$.
$H_n$ is called the
Gauss-Kronecker curvature. 
When the ambient space $M$ is a space form $M_c^{n+1}$ with constant sectional curvature $c$,
\begin{align*}&\text{Ric}=(n-1)cI+HA-A^2.\\
&R=\text{tr}\text{Ric}=n(n-1)c+H^2-|A|^2=n(n-1)c+2s_2.
\end{align*}
Hence $H_2$ is, modulo a constant, the
scalar curvature of $\Sigma$.

   In a local coordinate system, $g=(g_{ij})$, its inverse $g^{-1}=(g^{ij})$.  Given a smooth function $f$ on $\Sigma$,  $\text{Hess} f=\nabla^2f$ denotes the \text{\text{Hess}}ian of $f$ on $\Sigma$. Throughout  this paper, we use Einstein summation convention of summing over repeated indices. We  use the {\it raising and  
lowering indices} to change the type of tensor between a symmetric $(2,0)$ tensor and its corresponding $(1,1)$ tensor.
For instance,  $$g_i^j=g_{ik}g^{kj}=\delta_i^j,$$
$$(P_r)_{ij}=g_{ik}(P_r)^k_j.$$
\noindent We denote by $\langle\cdot , \cdot\rangle$ the inner product of two smooth tensor fields of the same type on $\Sigma$. Given  two smooth  symmetric $(2,0)$-tensor fields $(S_{ij})$ and $(T_{ij})$,  
$$\langle (S_{ij}),(T_{ij})\rangle=g^{ij}g^{kl}S_{ik}T_{jl}=S_i^lT_l^i=\langle (S_{i}^j),(T_{i}^j)\rangle.$$  
\noindent We have the notation $$|(T_{ij})|^2=\langle (T_{ij}),(T_{ij})\rangle=g^{ij}g^{kl}T_{ik}T_{jl}=T_i^jT_j^i=|(T_{i}^j)|^2.$$
 When there is no confusion, we omit writting the type of tensors, for instance $\langle S, T\rangle, |T|^2$.
We denote by $\text{tr}T$ and $\text{\r T}$ the trace of $(2,0)$ tensor $T$ and the traceless part  of $T$: $\text{\r T}=T-\frac{\text{tr}T}{n}g$ respectively. Then $$\text{tr}T=g^{ik}T_{ki}=T_i^i.$$
We use the notation: $\partial_t T_{ij}=\frac{\partial}{\partial t}(T_{ij}), \partial_t T_i^j=\frac{\partial}{\partial t}(T_{i}^j)$, which  define two tensors $(\partial_t T)_{ij}=\partial_t T_{ij}$ and $(\partial_t T)_{i}^j=\partial_t T_{ij}$, respectively.
\section{A normal deformation}\label{sec-var}

Let $F_0 : \Sigma^n\hookrightarrow (M^{n+1},\tilde{g})$ be a smooth immersion of a closed orientable hypersurface in a Riemannian manifold with the induced metric $g$.   We will prove the short-time existence of the the following initial value problem:  a one-parameter family  $F(\cdot,t): \Sigma\times [0,T)\to M^{n+1}$ of hypersurfaces $\Sigma_t=F(\cdot,t) $ satisfies:
\begin{equation}\label{evol-1}
\left\{
        \begin{array}{lll}
          \frac{\partial F}{\partial t}(x,t)&=f(x)\nu(x, t), & x\in\Sigma, t\in [0,T); \\
     F(x, 0)&=F_0(x ), &x\in \Sigma.
        \end{array}
      \right.
\end{equation}
where $f(x)$ is a smooth function on $\Sigma$ and $\nu(x,t)$ denotes the outer unit normal of $\Sigma_t$ at $F(x,t)$.

The approach is to represent the hypersurface $\Sigma_t$ as a graph in Fermi coordinates over the initial hypersurface $\Sigma$ and then to consider the deformation process as a  first order PDE equation for the height function. 

Let $\phi:\Sigma\times (-\epsilon,\epsilon)\to M^{n+1}$ be an  immersion, where $\epsilon$ is sufficiently small, by the exponential map
$$\phi(x,h)=\exp_{F_0(x)}h\nu(F_0(x)).$$
\noindent $\phi$ induces a metric $\phi^*\tilde{g}$ on $\Sigma\times (-\epsilon,\epsilon)$ from $(M^{n+1},\tilde{g})$. Let $x_1,\ldots, x_n$ denotes a local coordinate  on $\Sigma$. Then $\{\frac{\partial }{\partial x_1}, \ldots,\frac{\partial }{\partial x_n},\frac{\partial }{\partial h} \}$ is a local coordinate frame on $\Sigma\times (-\epsilon,\epsilon)$. By the Gauss lemma,  the metric $\phi^*\tilde{g}$ satisfies
$$(\phi^*\tilde{g})_{ih}=\tilde{g}\left(\frac{\partial }{\partial x_i},\frac{\partial }{\partial h}\right)=0, i=1,\ldots,n; (\phi^*\tilde{g})_{hh}=\tilde{g}\left(\frac{\partial }{\partial h},\frac{\partial }{\partial h}\right)=1.$$
\noindent  Clearly for $t$ fixed, $\phi_t(\Sigma)=\phi(\Sigma,t)$ is a hypersurface in $M^{n+1}$. If $\Sigma$ is embedded, $\phi$ gives the so-called Fermi-coordinates on a tubular neighborhood of $F_0(\Sigma)$. 
 
 Given a smooth function $u :\Sigma\to  (-\epsilon,\epsilon)$, the map $\psi:\Sigma \to \Sigma\times (-\epsilon,\epsilon)$ by $\psi(x)=(x,u(x))$ is an immersion, i.e.,  the graph $G(u)$ of $u$  is a hypersurface in $\Sigma\times (-\epsilon,\epsilon)$. A local coordinate frame on $G(u)$ is
 $$\frac{\partial\psi}{\partial x_{i}}= \frac{\partial}{\partial{x_i}}+\frac{\partial u}{\partial{x_i}}\frac{\partial}{\partial h},\quad  i=1,\ldots, n.$$
 The unit normal field ${\bf n}$ of $G(u)$ is
 $${\bf n}(x,u(x))=\frac{1}{W}\left( \frac{\partial}{\partial{h}}-(\phi^*\tilde{g})^{ij}\frac{\partial u}{\partial{x_i}} \frac{\partial}{\partial{x_j}}\right)(x,u(x)),$$
 where $W(x,u(x))=\sqrt{1+(\phi^*\tilde{g})^{ij}\frac{\partial u}{\partial {x_i}}\frac{\partial u}{\partial{x_j}}}(x,u(x))$.
 Hence $$\langle{\bf n}, \frac{\partial}{\partial{h}}\rangle =\frac{1}{W}, 
 \quad \frac{\partial}{\partial{h}}=\frac{1}{W}{\bf n}+\big(\frac{\partial}{\partial{h}}\big)^{\top},$$
  where $\big(\frac{\partial}{\partial h}\big)^\top$ is the projection of $\frac{\partial}{\partial{h}}$ to the tangent space of $G(u)$ spanned by $\frac{\partial\psi}{\partial x_{i}}, i=1,\ldots,n$.

 Under the above expression, we will prove that
 
\begin{thm} The initial value problem (\ref{evol-1})  has the unique smooth solution for $T$ sufficiently small.
\end{thm}
\begin{proof} Consider the initial value problem of the first order PDE on $\Sigma\times [0, T), T<\epsilon$:
\begin{equation}\label{evol-u}
\left\{
        \begin{array}{ll}
          \frac{\partial u}{\partial t}(x,t)&=f(x)\sqrt{1+(\phi^*\tilde{g})^{ij}\frac{\partial u}{\partial{x_i}}\frac{\partial u}{\partial{x_j}}(x,u(x,t))}, (x,t)\in \Sigma\times (0,T)\\
 u(x,0)&=0, \quad  x\in \Sigma.
        \end{array}
      \right.
\end{equation} 
By the theory of the first order PDE (cf \cite{Ca} \S 35--48; or \cite{E}, Chapter 3, Section 3), \eqref{evol-u} exists the unique smooth solution for $t<T$, where $T$ is sufficiently small. Here the closeness of  $\Sigma$ guarantees the global existence of the solution. Also we may choose $T$ sufficiently small so that $u(x,t)\in (-\epsilon,\epsilon)$.

By the solution $u(x,t)$ of \eqref{evol-u}, 
 we may construct a one-parameter family of hypersurfaces, i.e., graphs $G(u_t)$, where $u_t(x)=u(x,t)$, $t\in [0,T)$, parametrised by $\Psi: \Sigma\times [0,T)\to \Sigma\times (-\epsilon,\epsilon)$ satisfying $$\Psi(x,t)=\Psi_t(x) = (x, u(x, t)).$$  Then 
 $$\frac{\partial \Psi}{\partial t}=\frac{\partial u}{\partial t}\frac{\partial }{\partial h}=\frac{\partial u}{\partial t}\big(\frac{1}{W}{\bf n}\big)+\frac{\partial u}{\partial t}\big(\frac{\partial}{\partial{h}}\big)^{\top}.$$
 \noindent  So we have 
 \begin{equation}\frac{\partial \Psi}{\partial t}(x,t)=f(x){\bf n}(x,t)+\lambda(x,t),
 \end{equation}
 \noindent where ${\bf n}(x,t)$ denotes the unit normal of $\Psi_t(\Sigma)$ at $\Psi(x,t)$ and $\lambda=\frac{\partial u}{\partial t}\big(\frac{\partial}{\partial h}\big)^{\top}$ is the projection  of $\frac{\partial \Psi}{\partial t}$ to the tangent space of $\Psi_t(\Sigma)$ spanned by $\frac{\partial\Psi}{\partial x_{i}}, i=1,\ldots, n$.
 
 Let $\alpha: \Sigma \times [0,T)\to \Sigma$ be a smooth one-parameter  family of diffeomorphisms of $\Sigma$ satisfying
\begin{equation}\label{dif}\frac{\partial\alpha }{\partial t}=-\Psi_t^*\lambda, \quad \alpha_0=Id \quad (\text{Identity})
\end{equation}
Integrating \eqref{dif} directly, we obtain  the unique smooth  solution.
\noindent Define $\Phi(x,t)=\Psi(\alpha(x,t),t).$ Then
\begin{eqnarray*}
\frac{\partial \Phi}{\partial t}(x,t)&=&\Psi_*(\frac{\partial\alpha }{\partial t})(\alpha(x,t),t)+\big(\frac{\partial}{\partial t}\Psi(\cdot,t)\big)(\alpha(x,t),t)\\
&=&-\lambda\Psi(\alpha(x,t),t)+f(x){\nu}(\alpha(x,t),t)+\lambda \Phi(x,t)\\
&=&f(x){\bf n}(x,t).
\end{eqnarray*}
Finally note the exponential map $\phi$ is a local isometry. Define $F: \Sigma\times [0,T)\to M^{n+1}$ by $F=\phi\circ\Phi$.
 Then $F(x,t)$ is just the solution of Problem (\ref{evol-1}). 
 
 \end{proof}

\section{Evolution equations of geometric quantities  and  proof of Theorem \ref{opt-1}}\label{sec-mean}

In this section, we will prove Theorem \ref{opt-1}. Although there is a unified proof of it and Theorem \ref{opt-2}, we prefer to give an independent proof of Theorem \ref{opt-1}. One reason is that the evolution equations for mean curvature and the shape operator are much more simple than the ones for  $r$-th mean curvatures and the Newton transformations $P_r$, $2\leq r\leq n$.  We use the  method by Perez \cite{P} in proving Theorem \ref{thm-3}.

Let $F_0 : \Sigma^n\hookrightarrow (M^{n+1},\tilde{g})$ be a smooth immersion of a closed orientable hypersurface in a Riemannian manifold with the induced metric $g$.
Consider the normal deformation of hypersurfaces according to the  equation:
\begin{equation}\label{evol-g}
\left\{
        \begin{array}{lll}
          \frac{\partial F}{\partial t}(x,t)&=f(x,t)\nu(x, t), & x\in\Sigma, t\in [0,T); \\
     F(x, 0)&=F_0(x ), &x\in \Sigma.
        \end{array}
      \right.
\end{equation}
where $f(x,t)$ is a smooth function on $\Sigma$ and $\nu(x,t)$ denotes the outer unit normal of $\Sigma_t=F(\Sigma,t)$ at $F(x,t)$.

 If the solution of \eqref{evol-g} exists,   the following basic evolution equations holds under \eqref{evol-g} (cf \cite{H},  \cite{HP}):
\begin{prop}\label{basic} For any solution of \eqref{evol-g}, it  holds that
\begin{eqnarray}
 \partial_tg_{ij}&=&2f A_{ij},\\
  \partial_tg^{ij}&=&-2f A^{ij},\\
 \partial_tg_i^j&=&0\\
  \partial_t(dvol)&=&fHdvol,\\
  \partial_t{\bf \nu}&=&-\nabla f,\\
 \partial_tA_{ij}&=&-(\text{\text{\text{Hess}}}f)_{ij}+ f (A^2)_{ij}-f\widetilde{R}_{0i0j},\\
  \partial_tH&=&-\Delta f-|A|^2f-f\widetilde{\text{Ric}}(\nu,\nu).
  \label{h-evol}
  \end{eqnarray}
  \noindent In particular, when the ambient space is the space form $M_c^{n+1}$ with the sectional curvature $c$, 
\begin{eqnarray}
  \partial_tA_{ij}&=&-(\text{\text{\text{Hess}}}f)_{ij}+ f  (A^2)_{ij}-cfg_{ij},\\
   \partial_tH&=&-\Delta f-|A|^2f-ncf.
\end{eqnarray}
\end{prop}
 \noindent Here and thereafter, for simplicity, we drop the $t$ subscript wherever it would not lead to confusion. For instance, $g$and $A$ denote the induced metric  $g_t=F_t^*(\tilde{g})$ and the second fundamental form of $(\Sigma_t, g_t)$ respectively. 
   When  the ambient space is the space form $M_c^{n+1}$, 
 Proposition \ref{basic} yields 
\begin{prop} 
\begin{eqnarray}
  \partial_tA_i^j&=&-(\text{\text{\text{Hess}}} f+fA^2+cfg)_i^j.\label{a-ij}\\
 \partial_t|A|^2&=&-2\langle\text{\text{\text{Hess}}}f, A\rangle-2f\text{tr}A^3-2cfH.\\
\label{4-A}
\partial_t|\text{\r A}|^2&=&-2\langle\text{\text{\text{Hess}}}f,\text{\r A}\rangle-2f\langle A^2,\text{\r A}\rangle,
 \end{eqnarray}
 where $\text{\r A}=A-\frac{trA}{n}g.$
\end{prop}
\begin{proof}
\begin{eqnarray*}
 \partial_tA_i^j&=&\partial_t(A_{ik}g^{kj})\\
 &=&(\partial_tA_{ik})g^{kj}+A_{ik}\partial_tg^{kj}\\
 &=& [-(\text{\text{\text{Hess}}}f)_{ik}+ f  (A^2)_{ik}-cfg_{ik}]g^{kj}-2fA_{ik}A^{kj}\\
 &=& -(\text{\text{\text{Hess}}} f)_i^j -f(A^2)_i^j-cfg_i^j
\end{eqnarray*}
\begin{eqnarray*}
 \partial_t|A|^2&=&\partial_t(A_{i}^jA_j^i)\\
 &=&(\partial_tA_{i}^j)A_j^i+A_{i}^j(\partial_tA_j^{i})\\
 &=& -(\text{\text{\text{Hess}}} f+fA^2+cfg)_i^jA_j^i-A_i^j(\text{\text{\text{Hess}}} f+fA^2+cfg)_j^i\\
 &=& -2\langle\text{\text{\text{Hess}}}f, A\rangle-2f\text{tr}A^3-2cfH.
\end{eqnarray*}
\begin{eqnarray*}
 \partial_t|\text{\r A}|^2&=&\partial_t(|A|^2-\frac{H^2}{n})\\
 &=&\partial_t(|A|^2)-\frac{2H}{n}\partial_tH\\
 &=&  -2\langle\text{\text{\text{Hess}}}f, A\rangle-2f\text{tr}A^3-2cfH+\frac{2H}{n}(\Delta f+|A|^2f+ncf)\\
 &=& -2\langle\text{\text{\text{Hess}}}f, A\rangle+\frac{2H}{n}\langle\text{\text{\text{Hess}}}f, g\rangle-2f\langle A^2,A\rangle+\frac{2H}{n}\langle A^2,g\rangle\\
 &=&-2\langle\text{\text{\text{Hess}}}f,\text{\r A}\rangle-2f\langle A^2,\text{\r A}\rangle.
\end{eqnarray*}

\end{proof}
In the rest of this section, we  will consider the case that the ambient space is the space form $M_c^{n+1}$ and $f(x,t)=f(x) $, that is, the normal deformation \eqref{evol-1}. Assume that $\Sigma$ is a closed totally umbilical hypersurface in $M^{n+1}_c$. It is well known that $\Sigma$
must be a geodesic sphere, i.e, distance sphere $\mathbb{S}^n(a)$, where $a$ denotes its geodesic radius. Assume $F(x,t):\Sigma\times [0,T)\to M_c^{n+1}$ be the solution of the normal deformation \eqref{evol-1}. 
We consider the functional $\mathcal{F}: \Sigma_t=F(\Sigma,t)\to \mathbb{R}$ given by
\begin{equation}\label{f-1}
\mathcal{F}(\Sigma_t)=C^2\int_{\Sigma}|\text{\r{A}}|^2-\int_{\Sigma}\left(H-\overline{H}\right)^2,
\end{equation}
where $C$ is a constant and the subscripts $t$ of  $\text{\r A},H$ and $\overline{H}$ are omitted. 

Obviously, $\mathcal{F}(\Sigma)=0$. Next we obtain the first variation of $\mathcal{F}$ at $t=0$ as follows.
\begin{prop} \label{first-vari} If $\Sigma$ is totally umbilical, then 
$ \frac{d}{dt}\mathcal{F}(\Sigma_t)|_{t=0}=0$.
\end{prop}
\begin{proof}
\begin{eqnarray*}
\frac{d}{dt}\mathcal{F}(\Sigma_t)
&=&C^2\int_{\Sigma}\partial_t|\text{\r{A}}|^2dvol+C^2\int_{\Sigma}|\text{\r{A}}|^2\partial_t(dvol)\\
& &-\int_{\Sigma}\partial_t(H-\overline{H})^2dvol
-\int_{\Sigma}\left(H-\overline{H}\right)^2\partial_t(dvol).
\end{eqnarray*}
At   $t=0$, 
\begin{equation}\label{a1-trace}
\text{\r A}=0. 
\end{equation}
By \eqref{ine-i-03},
\begin{equation}\label{h-con}
 H=\overline{H}.
\end{equation}
\begin{equation}\label{h-1}
\partial_t(H-\overline{H})^2|_{t=0}=2(H-\overline{H})\partial_t(H-\overline{H})|_{t=0}=0.
\end{equation}
By \eqref{4-A}, or directly 
\begin{eqnarray}\label{a2-trace}
\partial_t|\text {\r A}|^2|_{t=0}&=&\partial_t (\text {\r A}_{i}^j\text {\r A}_{j}^i)|_{t=0}\nonumber\\
&=&(\partial_t \text {\r A}_{i}^j)\text {\r A}_{j}^i|_{t=0}+ \text {\r A}_{i}^j(\partial_t\text {\r A}_{j}^i)|_{t=0}\nonumber\\
&=&0.
\end{eqnarray}
Hence $\frac{d}{dt}\mathcal{F}(\Sigma_t)|_{t=0}=0$.
\end{proof}

Furthermore we discuss  the second variation of $\mathcal{F}$. The straightforward computation implies  the following conclusion.
\begin{prop} \label{varphi-1} Suppose $\varphi(\cdot,t): \Sigma\times (-\epsilon,\epsilon)\to \mathbb{R}$ is a  smooth function. If $\varphi |_{t=0}=0$ and $ \partial_t\varphi |_{t=0}=0$, then 
\begin{equation} 
\bigg(\frac{d^2}{dt^2}\int_{\Sigma}\varphi\bigg)(0)=\int_{\Sigma}\partial_t^2\varphi |_{t=0}.
\end{equation}
\end{prop}
\noindent Take $\phi=|\text{\r A}|^2$ and $ (H-\overline{H})^2$ in Proposition \ref{varphi-1} respectively. By \eqref{a1-trace}, \eqref{h-con}, \eqref{h-1} and  \eqref{a2-trace}, it holds that 
   \begin{equation}\label{second-vari-1}
\frac{d^2}{d t^2}F(\Sigma_t)|_{t=0}=C^2\int_{\Sigma}\partial_t^2|\text{\r A}|^2|_{t=0}-\int_{\Sigma}\partial_t^2(H-\overline{H})^2|_{t=0}
\end{equation}
Now we will calculate the right hand of (\ref{second-vari-1}).
\begin{eqnarray}\label{4-A-1}
\nonumber
&&\partial_t\text{\r A}_i^j\\
\nonumber
&=&\partial_t(A_i^j-\frac Hng_i^j)\\
\nonumber
&=&\partial_tA_i^j-\frac 1n(\partial_tH)g_i^j-\frac{H}{n}\partial_tg_i^j\\
\nonumber
&=&-[\text{\text{\text{Hess}}} f+fA^2+fcg]_i^j+\frac 1n (\Delta f+|A|^2f+ncf)g_i^j\\
\nonumber
 &=&-(\text{\text{\text{Hess}}} f)_{i}^j+\frac 1n(\Delta f)g_i^j-f(A^2)_i^j+\frac 1n|A|^2fg_i^j\\
 &=&-(\text{\text{\text{Hess}}}f)_{i}^j+\frac 1n(\Delta f)g_i^j-f(A\text{\r A})_i^j+\frac 1n|\text{\r A}|^2fg_i^j-\frac1n Hf\text{\r A}_i^j.
\end{eqnarray}
In the last equality of \eqref{4-A-1}, we used the identity
$$A^2-\frac{1}{n}|A|^2g=A\text{\r A}+\frac{H}{n}A-\frac{1}{n}(|\text{\r A}|^2+\frac{H^2}{n})g=fA\text{\r A}-\frac{1}{n}|\text{\r A}|^2+\frac{H}{n}\text{\r A}.$$
Note $\text{\r A}|_{t=0}=0$. By \eqref{4-A-1},
\begin{eqnarray*}
\partial_t\text{\r A}_i^j|_{t=0}&=&-(\text{\text{\text{Hess}}}f)_{i}^j|_{t=0}+\frac 1n(\Delta f)g_i^j|_{t=0}\\
&=&-(\text{\r Hess}f)_i^j|_{t=0},
\end{eqnarray*}
and 
\begin{eqnarray}\label{4-A-2}
\nonumber
\partial_t^2|\text{\r A}|^2|_{t=0}&=&\partial_t[\partial_t(\text{\r A}_i^j\text{\r A}_j^i)]|_{t=0}\\
\nonumber
&=&[(\partial_t^2\text{\r A}_i^j)\text{\r A}_j^i+2(\partial_t\text{\r A}_i^j)(\partial_t(\text{\r A}_j^i)+\text{\r A}_i^j(\partial_t^2\text{\r A}_j^i)]|_{t=0}\\
&=&2(\partial_t\text{\r A}_i^j|_{t=0})(\partial_t\text{\r A}_j^i|_{t=0})\\
\nonumber
&=&2(\text{\r Hess}f)_i^j|_{t=0}(\text{\r Hess}f)_j^i|_{t=0}\\
\nonumber
&=&2|\text{\r Hess}f|^2|_{t=0}
\end{eqnarray}
In order to calculate the second term of the right side of \eqref{second-vari-1}, we need the following proposition, which was proved in \cite{P}.
\begin{prop} \cite{P}\label{average} Suppose $\varphi(\cdot,t): \Sigma\times (-\epsilon,\epsilon)\to \mathbb{R}$ be a smooth function. Let $\overline{\varphi}(t)=\frac{1}{\text{vol}(\Sigma_t)}\int_{\Sigma}\varphi(t)dvol(\Sigma_t)$ be average of $\varphi(t)$ on $\Sigma_t$. Then 
\begin{equation}
\frac{d}{dt}\overline{\varphi}=\overline{\partial_t\varphi}+\overline{fH(\varphi-\overline{\varphi})}.
\end{equation}

\end{prop}
\noindent For the completeness of proof, we include its proof here.

\begin{proof}
\begin{eqnarray*}
&&\partial_t\overline{\varphi}\\
&=&\big[\partial_t(\frac{1}{\text{vol}(\Sigma_t)})\big]\int_{\Sigma}\varphi dvol+\frac{1}{\text{vol}(\Sigma_t)}\partial_t\big(\int_{\Sigma}\varphi dvol\big)\\
&=&-\frac{1}{\text{vol}(\Sigma_t)^2}\big[\partial_t(\int_{\Sigma}dvol)\big]\int_{\Sigma}\varphi dvol
\\
&& \quad +\frac{1}{\text{vol}(\Sigma_t)}\int_{\Sigma}\partial_t\varphi dvol+\frac{1}{\text{vol}(\Sigma_t)}\int_{\Sigma}\varphi \partial_t(dvol)\\
&=&-\frac{1}{\text{vol}(\Sigma_t)}\big(\int_{\Sigma}fH\overline{\varphi}dvol\big)+\overline{\partial_t\varphi}+\frac{1}{\text{vol}(\Sigma_t)}\big(\int_{\Sigma}fH\varphi dvol\big)\\
&=&\overline{\partial_t\varphi}+\overline{fH(\varphi-\overline{\varphi})}.
\end{eqnarray*}

\end{proof}
\noindent Taking $\varphi=H$ in Proposition \ref{average}, we have
\begin{align}
&\partial_t\overline{H}=\overline{\partial_t H}+\overline{fH(H-\overline{H})}.
\end{align}
At $t=0$, by \eqref{h-con}, \eqref{h-evol} and $|A|=\frac {H^2}{n}$, it holds that
\begin{eqnarray*}
\partial_t(H-\overline{H})|_{t=0}&=&(\partial_tH-\partial_t\overline{H})|_{t=0}\\
&=&(\partial_tH-\overline{\partial_t H}-\overline{fH(H-\overline{H})})|_{t=0}\\
&=&(\partial_tH-\overline{\partial_t H})|_{t=0}\\
&=&-\Delta f-|A|^2f-ncf+\overline{ \Delta f+|A|^2f+ncf}\\
&=&-\Delta f-\frac{H^2}{n}f-ncf+\overline{ \Delta f+\frac{H^2}{n}f+ncf}\\
&=&-\Delta f-n\omega f+\overline{ \Delta f+n\omega f},
\end{eqnarray*}
where $\omega=\frac {H^2}{n^2}+c$. So
\begin{eqnarray}\label{4-H}
\nonumber
\partial_t^2(H-\overline{H})^2|_{t=0}&=&
2[\partial_t(H-\overline{H})]^2+2(H-\overline{H})\partial_t^2(H-\overline{H})\\
&=&2[\partial_t(H-\overline{H})]^2|\\
\nonumber
&=&2 (\Delta f+n\omega f-\overline{ \Delta f+n\omega f})^2
\end{eqnarray}
\noindent Thus   \eqref{second-vari-1} together with \eqref{4-A-2} and \eqref{4-H} yields that
\begin{prop}\label{second-vari-2} Let $\Sigma$ is a totally umbilical hypersurface in the space form.  It holds that
\begin{equation}
\frac{d^2}{dt^2}\mathcal{F}(\Sigma_t)|_{t=0}=2C^2\int_{\Sigma}|\text{\r Hess}f|^2(0)-2\int_{\Sigma}(\Delta f+n\omega f-\overline{ \Delta f+n\omega f})^2.
\end{equation}
\end{prop}

In the following, we will show that 
\begin{thm} \label{thm-2der}Let $M_c^{n+1}$  be a space form of the constant sectional curvature $c$. Then there exists a closed  totally umbilical hypersurface $\Sigma$ in $M$ and a deformation $F(\cdot,t)$ of $\Sigma$ such that 
$$\frac{d^2}{dt^2}\mathcal{F}(\Sigma_t)|_{t=0}<0,$$ where 
$\mathcal{F}(\Sigma_t)$ is given by \eqref{f-1} with constant $ C<\sqrt{\frac{n}{n-1}}$.

\end{thm}
\begin{proof}
We first consider the case of  the simply connected space forms $M_c^{n+1}$,  i.e., the Euclidean space $\mathbb{R}^{n+1}, c=0$, the  Euclidean sphere $\mathbb{S}^{n+1}, c>0$, and the hyperbolic space  $\mathbb{H}^{n+1}, c<0$ respectively.  Here for convenience we take the Poincar\'e model for $\mathbb{H}^{n+1}$.  Given an $M_c^{n+1}$, its rotationally symmetric metric is denoted by    $$\tilde{g}=dr^2+sn_c^2( r)\eta,$$ 
where $\eta$ denotes the metric of the unit Euclidean sphere $\mathbb{S}^n$, $r$ denotes the distance  under the metric $\tilde{g}$ to the pole $o$ and $sn_c( r)$ is a function given by
$$sn_c( r)=\left\{
        \begin{array}{ll}
          r, & \hbox{if $c=0$}; \\
          \frac{\sin(\sqrt{c}r)}{\sqrt{c}}, & \hbox{if $c>0$};\\
           \frac{\sinh(\sqrt{|c|}r)}{\sqrt{|c|}}, & \hbox{if $c<0$}.
        \end{array}
      \right.$$
Now fix a number $a>0$  (in the case of $\mathbb{S}^{n+1}$, $0<a<\frac{\pi}{\sqrt{c}}$) and choose $\Sigma$ as the geodesic sphere $\mathbb{S}^n(a)
$ in $M$ with the  geodesic radius $a$ centered at $o$. It is well known that $\Sigma$ is totally umbilical. On the other hand, $\Sigma$ has the induced metric $g=sn_c^2( a)\eta$. The metric $g$ is the metric of the round sphere with the radius $sn_c(a)$ and so the Ricci curvatue of $\Sigma$ is $$Ric_{\Sigma}(\nabla f,\nabla f)=\frac {n-1}{sn^2_c( a)}|\nabla f|^2.$$
\noindent Recalling the Bochner formula
$$\frac12\Delta|\nabla f|^2=|\text{\text{\text{Hess}}} f|^2+\text{Ric}^{\Sigma}(\nabla f,\nabla f)+\left<\nabla f,\nabla(\Delta f)\right>,$$
and integrating it,  by the Stokes'  formula, we have
\begin{eqnarray*}
\int_{\Sigma}|\text{\text{\text{Hess}}} f|^2&=&\int_{\Sigma}(\Delta f)^2-\int_{\Sigma}\text{Ric}_{\Sigma}(\nabla f,\nabla f)\\
&=&\int_{\Sigma}(\Delta f)^2-\frac {n-1}{sn^2_c( a)} \int_{\Sigma}|\nabla f|^2
\end{eqnarray*}
\begin{eqnarray*}
\int_{\Sigma}|\text{\r Hess} f|^2&=&\int_{\Sigma}|\text{\text{\text{Hess}}} f|^2-\frac{1}{n}\int_{\Sigma}(\Delta f)^2\\
&=&\frac{n-1}{n}\int_{\Sigma}(\Delta f)^2-\frac {n-1}{sn^2_c( a)} \int_{\Sigma}|\nabla f|^2\\
&=&\frac{n-1}{n}\int_{\Sigma}(\Delta f)^2+\frac {n-1}{sn^2_c( a)} \int_{\Sigma}f\Delta f.
\end{eqnarray*}
Note that the closeness of $\Sigma$ implies the  average $\overline{\Delta f}=0$. Assume that  $f$ satisfies $\int_{\Sigma}f=0$ (such $f$ will be chosen later). Then 
$$\overline{ \Delta f+n\omega f}=\overline{ \Delta f}+n\omega\overline{ f}=0$$ and so
\begin{eqnarray*}
\int_{\Sigma}\left(\Delta f+n\omega f-\overline{ \Delta f+n\omega f}\right)^2&=&\int_{\Sigma}(\Delta f+n\omega f)^2\\
&=&\int_{\Sigma}(\Delta f)^2+2n\omega\int_{\Sigma}f\Delta f+n^2\omega^2\int_{\Sigma}f^2.
\end{eqnarray*}
Thus 
\begin{eqnarray*}
&&\frac {1}{2}\frac{d^2}{dt^2}\mathcal{F}(\Sigma_t)|_{t=0}\\
&=&\frac{C^2(n-1)}{n}\int_{\Sigma}(\Delta f)^2+\frac{C^2(n-1)}{sn^2_c( r) }\int_{\Sigma}f\Delta f\\
&  &\qquad -\int_{\Sigma}(\Delta f)^2-2n\omega\int_{\Sigma}f\Delta f-n^2\omega^2\int_{\Sigma}f^2\\
&=&[\frac{C^2(n-1)}{n}-1]\int_{\Sigma}(\Delta f)^2+[\frac{C^2(n-1)}{sn^2_c( a)}-2n\omega]\int_{\Sigma}(f\Delta f)-n^2\omega^2\int_{\Sigma}f^2\\
\end{eqnarray*}
Now we choose  $f$  to be an eigenfunction of the Laplacian on $(\Sigma,g)$ corresponding to the nonzero eigenvalue $ \xi(k)$, that is,
$$\Delta_gf=-\xi(k)f, \quad \int_{\Sigma}f=0.$$
 It is known that on $\Sigma$,
$$\Delta_g=\frac{1}{sn^2_c( a)}\Delta_{\eta}.$$
 Hence the nonzero eigenvalues $\xi(k)$ are 
$$\xi(k)=\frac{k(k+n-1)}{sn^2_c( a)}, k=1,2,\ldots$$
The sequence $\xi(k)$ increases  and diverges to the $+\infty$ as $k$ tends to $+\infty$. For such $f$,
\begin{align*}
& \frac {1}{2}\frac{d^2}{dt^2}\mathcal{F}(\Sigma_t)|_{t=0}\\
&=\big(\frac{C^2(n-1)}{n}-1\big)\int_{\Sigma}(\Delta f)^2+\big(\frac{C^2(n-1)}{sn^2_c( a)}-2n\omega\big)\int_{\Sigma}(f\Delta f)-n^2\omega^2\int_{\Sigma}f^2\\
&=\big(\frac{C^2(n-1)}{n}-1\big)\xi(k)^2\int_{\Sigma}f^2-\big(\frac{C^2(n-1)}{sn^2_c( a)}-2n\omega\big)\xi(k)\int_{\Sigma}f^2-n^2\omega^2\int_{\Sigma}f^2\\
&=\bigg[\bigg(\frac{C^2(n-1)}{n}-1\bigg)\xi(k)^2-\bigg(\frac{C^2(n-1)}{sn^2_c( a)}-2n\omega\bigg)\xi(k)-n^2\omega^2\bigg]\int_{\Sigma}f^2\\
\end{align*}
When $C<\sqrt{\frac{n-1}{n}}$, the coefficient of $\xi^2(k)$ in the last equality is negative. Hence, if $\xi$ big enough,  the quadratic  polynomial  $\frac {1}{2}\frac{d^2}{dt^2}\mathcal{F}(\Sigma_t)|_{t=0}$ is negative. By the property of $\xi(k)$, there exists a $k_0$ sufficiently large so that $k\geq k_0$, $\frac {1}{2}\frac{d^2}{dt^2}\mathcal{F}(\Sigma_t)|_{t=0}<0$. 
This is the case of the  simply connected $M_c^{n+1}$. Observe that the geodesic radius $a$ of the geodesic sphere $\Sigma$ can be arbitrarily chosen only if it  makes sense.

If $M_c^{n+1}$ is not simply connected, we consider  the universal covering map $\pi: \tilde{M}\to M_c^{n+1}$.  It is known that $\pi$ is a local isometry. Let $\Omega\subset \tilde{M}$ be a neighborhood of the pole $o\in \tilde{M}$ such that $\pi:\Omega\to \pi(\Omega)\subset M$ is an isometry. In $\Omega$, by the conclusion of the simply connected space form,  there exists a geodesic sphere $\tilde{\Sigma}\subset\Omega$ (with small geodesic radius) and a deformation $\tilde{F}(\cdot,t):\tilde{\Sigma}\times(-\epsilon,\epsilon)\to \mathbb{R}$, which has the conclusion of the theorem.  Take $\Sigma=\pi(\tilde{\Sigma})$ and the deformation  $F(\cdot,t)=\tilde{F}(\pi^{-1}(\cdot),t)$. Then  $\Sigma$ and $\mathcal{F} $ satisfy  the theorem.

\end{proof}

\noindent {\it {The proof of Theorem \ref{opt-1}}}. For the deformation $F(x,t)$ given in Theorem \ref{thm-2der} with $\Sigma_0=\mathbb{S}^n$, the functional $\mathcal{F}$ satisfies
\begin{equation}\label{deri}
\mathcal{F}(\Sigma_t)|_{t=0}=\frac{d}{dt}\mathcal{F}(\Sigma_t)|_{t=0}=0, \quad\frac{d^2}{dt^2}\mathcal{F}(\Sigma_t)|_{t=0}<0.
\end{equation}
 \eqref{deri}  implies that $\mathcal{F}(\Sigma_t)<0$
 for $t$ sufficiently small, that is, it  holds on $\Sigma_t$:
 \begin{eqnarray}\label{ine-opt1-2}
\int_{\Sigma_t}\left(H-\overline{H}\right)^2&>& C^2\int_{\Sigma_t}|A-\frac{H}{n}g|^2.
 \end{eqnarray} 
  By \eqref{ine-opt1-2} and the identity:
 \begin{equation}
|A-\frac{\overline{H}}{n}g|^2= |A-\frac{H}{n}g|^2+\frac{1}{n}\left(H-\overline{H}\right)^2,
 \end{equation} 
 we have that for $C<\sqrt{\frac{n}{n-1}}$,
 \begin{eqnarray}
\nonumber
\int_{\Sigma_t}|A-\frac{\overline{H}}{n}g|^2&>& \left(1+\frac{C^2}{n}\right)\int_{\Sigma_t}|A-\frac{H}{n}g|^2\\
&>&C^2\int_{\Sigma_t}|A-\frac{H}{n}g|^2.
 \end{eqnarray} 
 Since $\Sigma_{t}$ is arbitrarily close to $\mathbb{S}^n(a)$, the Ricci curvature $\text{Ric}_{\Sigma_t}>0$. 
 So  we complete the proof  of theorem.

\qed

\section{Proof of Theorem \ref{opt-2}}\label{sec-rmean}

In this section, we first give the needed evolution equation of $s_r$ (roughly $r$-th mean curvatures) under the general normal deform. Next we prove Theorem \ref{opt-2}. For $r\geq 2$,  instead of calculating the complicated evolution equations of $P_r$, we compute directly the  corresponding values at $t=0$ by using the fact that $\Sigma=\Sigma_0$ is  totally umbilical (see \eqref{pr-evol-0}). 

Consider the normal deformation $F(x,t)$ of hypersurfaces in \eqref{evol-g}. 
  Recall a result proved by Reilly \cite{R}.
      \begin{prop}\cite{R}\label{reilly} Let $B = B(t)$  be a smooth one-parameter family of diagonalizable linear transformation of the vector space $V$,  $\sigma_r$ the symmetric functions of the eigenvalues of $B$ and $Q_r$  the Newton transformation with respect to $B$. Then for $r = 0, 1, \ldots, n$ we have
$$\partial_t\sigma_{r+1} = tr((\partial_tB)Q_r).$$
\end{prop}
\noindent Applying Proposition \ref{reilly} to the shape operator $A$, by \eqref{a-ij},  we have 
\begin{eqnarray*}
& &\frac{\partial }{\partial t}(s_r)\\
&=&\text{tr}((\partial_tA)P_{r-1})\\
 &=&-\text{tr}(P_{r-1}\text{Hess} f)-f\text{tr}(P_{r-1}A^2)-cf\text{tr}P_{r-1}.
\end{eqnarray*}
\noindent By \eqref{tr-1} and \eqref{tr-3},
\begin{eqnarray*}
 \frac{\partial }{\partial t}(s_r)&=&-\text{tr}(P_{r-1}\text{\text{Hess}} f)-f\big(s_1s_r-(r+1)s_{r+1}\big)-c(n-r+1)s_{r-1}.
 \end{eqnarray*}
So it holds that
\begin{cor}\cite{R} Under \eqref{evol-g},
\begin{eqnarray}
&&\frac{\partial }{\partial t}(s_r)\nonumber\\
 &=&-\text{tr}(P_{r-1}\text{Hess} f)-f\text{tr}(P_{r-1}A^2)-cf\text{tr}P_{r-1}\label{evol-s}\\
 &=&-\text{tr}(P_{r-1}\text{Hess} f)-f\big(s_1s_r-(r+1)s_{r+1}\big)-c(n-r+1)s_{r-1}.
\end{eqnarray}
\end{cor}

 To show Theorem \ref{opt-2}, we will use the same approach as in the proof of Theorem \ref{opt-1}. In the rest of this section, we assume that $F(x,t):\Sigma\times [0,T)\to M^{n+1}_c$ is  the solution of the normal deformation  \eqref{evol-1} and $\Sigma$ is a totally umbilical hypersurface.
  Define the functional $\mathcal{G}: \Sigma_t=F(\Sigma,t)\to \mathbb{R}$ given by
\begin{equation}\label{f-2}
\mathcal{G}(\Sigma_t)=C^2\int_{\Sigma}|\text {\r P}_r|^2-\int_{\Sigma}\left(s_r-\overline{s_r}\right)^2,
\end{equation}
where $C$ is a constant and the subscripts $t$ of  $\text {\r P}_r, s_r$ and $\overline{s_r}$ are omitted.

  First we need some combinatorial identities.
\begin{prop}\label{combin}
\begin{eqnarray}
\sum_{i=0}^{r} (-1)^i \binom{n}{r-i} &= &\binom{n-1}{r}\label{combin-1}\\
\binom{n+1}{r} &=& \binom{n}{r} + \binom{n}{r-1}\label{combin-2}\\
n \binom{n-1}{r} &=& (n-r) \binom{n}{r}\label{combin-3}\\
\sum_{i=0}^{r} (-1)^i \binom{n}{r-i}i &=& - \binom{n-2}{r-1}\label{combin-4}
\end{eqnarray}
\end{prop}
\eqref{combin-1}, \eqref{combin-2} and \eqref{combin-3} are well known. Since we couldn't find the adequate reference for \eqref{combin-4},  for the completeness of the proof, we prove \eqref{combin-4} here.
\begin{proof}
\begin{eqnarray}
\nonumber
&&\sum_{i=0}^{r} (-1)^i \binom{n}{r-i}i\\
 &=& -\sum_{i=0}^{r} (-1)^i \binom{n}{r-i}(r-i) + \sum_{i=0}^{r} (-1)^i \binom{n}{r-i}r\nonumber\\
\nonumber
&=& -\sum_{i=0}^{r-1} (-1)^i \frac{n(n-1)\cdots [n-(r-i)+1]}{(r-i)!}(r-i) + r \sum_{i=0}^{r} (-1)^i \binom{n}{r-i}\\
\nonumber
&=& -n \sum_{i=0}^{r-1} (-1)^i \frac{(n-1)\cdots [n-1-(r-i-1)+1]}{(r-i-1)!} + r \binom{n-1}{r}\\
\nonumber
&=& -n \sum_{i=0}^{r-1} (-1)^i \binom{n-1}{r-i-1} + r \binom{n-1}{r}\\
\nonumber
&=& -n \binom{n-2}{r-1} + r \binom{n-1}{r}\\
\nonumber
&=& -n \frac{n-r}{n-1} \binom{n-1}{r-1} + (n-r) \binom{n-1}{r-1}\\
&=&  -\frac{n-r}{n-1} \binom{n-1}{r-1}.\label{combin-4-4} 
\end{eqnarray}
In the verification of \eqref{combin-4-4}, we used \eqref{combin-1}. By \eqref{combin-3} and \eqref{combin-4-4},
\begin{equation}
\sum_{i=0}^{r} (-1)^i \binom{n}{r-i}i = - \binom{n-2}{r-1}.
\end{equation}
\end{proof}

Recall that  $\Sigma$ is a closed totally umbilical hypersurface, i.e., a geodesic sphere $\mathbb{S}^n(a)$ with the geodesic radius $a$  in the space form  $M_c^{n+1}$. Denote by $\lambda$ the principle curvatures of $\Sigma$. $\lambda$ is a constant.
So  for $\Sigma=\Sigma_0$,
\begin{eqnarray}
H&=&n\lambda,  \\
 A&=&\frac{H}{n}I=\lambda I,\\
A^k&=&\lambda^kI, \\
s_r&=&\displaystyle\sum_{i_1<\cdots<i_r}\lambda_{i_1}\ldots \lambda_{i_r}=\displaystyle\sum_{i_1<\cdots<i_r}\lambda^r=\binom{n}{r}\lambda^r,\label{sr-exp}\\
 \overline{s_r}&=&s_r.\label{sr-const}
\end{eqnarray}
     Moreover, by the definition \eqref{def-pr} of $P_r$ and \eqref{combin-1}, for $1\leq r\leq n$,
    \begin{eqnarray} 
     P_r&=&\displaystyle\sum_{j=0}^r(-1)^js_{r-j}A^j\nonumber\\
     &=&\displaystyle\sum_{j=0}^r(-1)^j\binom{n}{r-j}\lambda^{r-j}\lambda^jI\nonumber\\
     &=&\lambda^r\binom{n-1}{r}I.\label{id-pr}
\end{eqnarray}
By \eqref{tr-1}, \eqref{combin-3}, \eqref{sr-exp}, and  \eqref{id-pr}.
      \begin{eqnarray}
      \text{\r P}_r&=&0. \label{traceless}
     \end{eqnarray}
       \noindent     So \eqref{traceless} and \eqref{sr-const} imply that $\mathcal{G}(\Sigma_t)|_{t=0}=0$. 
The first variation of $\mathcal{G}$ at $t=0$ can be calculated as follows.
\begin{eqnarray}
\partial_t(|\text {\r P}_r|^2)|_{t=0}&=&\partial_t[ (\text {\r P}_r)_{i}^j(\text {\r P}_r)_{j}^i]|_{t=0}\nonumber\\
&=&[\partial_t (\text {\r P}_r)_{i}^j](\text {\r P}_r)_{j}^i|_{t=0}+ (\text {\r P}_r)_{i}^j\partial_t[(\text {\r P}_r)_{j}^i]|_{t=0}\nonumber\\
&=&0.\label{pr-evol-1}
\end{eqnarray}
\begin{eqnarray}
\partial_t(s_r-\overline{s_r})^2|_{t=0}=2(s_r-\overline{s_r})\partial_t(s_r-\overline{s_r})|_{t=0}=0.\label{sr-evol-2}
\end{eqnarray}
So
\begin{eqnarray*}
\frac{d}{dt}\mathcal{G}(\Sigma_t)|_{t=0}&=&\big[C^2\int_{\Sigma}|\text {\r P}_r|^2\partial_t (dvol)+C^2\int_{\Sigma}\partial_t(|\text {\r P}_r|^2)dvol\\
& &-\int_{\Sigma}\partial_t(s_r-\overline{s_r})^2dvol-\int_{\Sigma}(s_r-\overline{s_r})^2\partial_t(dvol)\big]|_{t=0}\\
&=&0.
\end{eqnarray*}
We obtain that
\begin{prop} For  $\Sigma$, 
$\frac{d}{dt}\mathcal{G}(\Sigma_t)|_{t=0}=0$.
\end{prop}
Next we  calculate the second variation of $\mathcal{G}(\Sigma_t)$ at $t=0$.
By \eqref{sr-const}, \eqref{traceless}, \eqref{pr-evol-1}  and \eqref{sr-evol-2},  Proposition \ref{average} implies that 
\begin{eqnarray}\label{r-secvar-1}\frac{d^2}{dt^2}\mathcal{G}(\Sigma_t)|_{t=0}=C^2\int_{\Sigma}\partial^2_t|\text {\r P}_r|^2|_{t=0}-\int_{\Sigma}\partial_t^2(s_r-\overline{s_r})^2|_{t=0}.
\end{eqnarray}
We have 
\begin{eqnarray*}\partial_t(|\text {\r P}_r|^2)=[\partial_t(\text {\r P}_r)_i^j](\text {\r P}_r)_j^i+(\text {\r P}_r)_i^j[\partial_t(\text {\r P}_r)_i^j],
\end{eqnarray*}
\begin{eqnarray*}\partial_t^2(|\text {\r P}_r|^2)=[\partial_t\partial_t(\text {\r P}_r)_i^j](\text {\r P}_r)_j^i+2[\partial_t(\text {\r P}_r)_i^j][\partial_t(\text {\r P}_r)_j^i]+(\text {\r P}_r)_i^j[\partial_t\partial_t(\text {\r P}_r)_i^j].
\end{eqnarray*}
So
\begin{eqnarray}\label{pr-1}\partial_t^2(|\text {\r P}_r|^2)|_{t=0}=2[\partial_t(\text {\r P}_r)_i^j][\partial_t(\text {\r P}_r)_j^i]_{t=0}=2|\partial_t\text {\r P}_r|^2_{t=0},
\end{eqnarray}
where the $(1,1)$-tensor $\partial_tP_r$ is defined by $(\partial_tP_r)_i^j:=\partial_t[(P_r)_i^j]$. On the other hand, we have 
\begin{eqnarray*}
\partial_t(s_r-\overline{s_r})^2=
2(s_r-\overline{s_r})\partial_t(s_r-\overline{s_r}).
\end{eqnarray*}
\begin{eqnarray*}\partial_t^2(s_r-\overline{s_r})^2=
2(s_r-\overline{s_r})\partial_t^2(s_r-\overline{s_r})+2[\partial_t(s_r-\overline{s_r})]^2.\end{eqnarray*}
\begin{eqnarray}\label{sr-1}\partial_t^2(s_r-\overline{s_r})^2(0)=2[\partial_t(s_r-\overline{s_r})]^2(0).\end{eqnarray}
So \eqref{r-secvar-1}, \eqref{pr-1} and \eqref{sr-1} imply that

\begin{prop}
\begin{eqnarray}\label{r-secvar-2}\frac{d^2}{dt^2}\mathcal{G}(\Sigma_t)|_{t=0}=2C^2\int_{\Sigma}(\partial_t|\text {\r P}_r|)^2|_{t=0}-2\int_{\Sigma}[\partial_t(s_r-\overline{s_r})]^2|_{t=0}.
\end{eqnarray}
\end{prop}
By \eqref{tr-1}, \eqref{evol-s}, \eqref{sr-exp} and  \eqref{id-pr},
\begin{eqnarray}
& &\frac{\partial }{\partial t}(s_r)|_{t=0}\nonumber\\
 &=&-\lambda^{r-1}\binom{n-1}{r-1}\text{tr} \text{\text{Hess}} f|_{t=0}-f\lambda^{r-1}\binom{n-1}{r-1}\lambda^2n-ncf\lambda^{r-1}\binom{n-1}{r-1}\nonumber\\
 &=&-\binom{n-1}{r-1}\lambda^{r-1}(\Delta f+n\lambda^2f+ncf). \label{var-sr1}
\end{eqnarray}
 Note on $\Sigma$,  $A=\lambda I$. Hence on $\Sigma$, $A(\text{\text{\text{Hess}}}f)=(\text{\text{\text{Hess}}}f) A$. This property let us prove the following   conclusion.
        
      \begin{prop} \label{a}
      \begin{eqnarray}
      \partial_t(A^m)_i^j|_{t=0}&=&-m[(\text{\text{\text{Hess}}}  f) A^{m-1}+fA^{m+1}+cf(A^{m-1})]_i^j|_{t=0}\label{a-1}\\
      &=&-m\lambda^{m-1}[\text{\text{\text{Hess}}}  f +\lambda^{2}f I+cf I]_i^j|_{t=0}\label{a-2}
      \end{eqnarray}
      \end{prop}  
        
        \begin{proof} We give the argument  by induction.
The conclusion holds for $m=1$.
Suppose $\partial_t (A^m)_i^j|_{t=0} = -m[(\text{\text{\text{Hess}}}  f )A^{m-1}+fA^{m+1}+cf(A^{m-1})]_i^j|_{t=0}$.  Then
\begin{eqnarray*}
\nonumber
& &\partial_t (A^{m+1})_i^j |_{t=0}\\
 &=& \partial_t [A_i^k(A^m)_k^j] |_{t=0}\\
\nonumber
&=& (\partial_t A_i^k) (A^m)_k^j |_{t=0}+ A_i^k  [\partial_t (A^m)_k^j]|_{t=0}\\
\nonumber
& = & -[\text{\text{Hess}} f +fA^2+cf I]_i^k (A^m)_k^j |_{t=0} \\
\nonumber &&\quad -mA_i^k[(\text{\text{\text{Hess}}}  f) A^{m-1}+fA^{m+1}+cf(A^{m-1})]_k^j|_{t=0}\\
\nonumber
& = & -(m+1)[(\text{\text{\text{Hess}}} f )A^{m}+fA^{m+2}+cf(A^{m})]_i^j|_{t=0}.
\end{eqnarray*}
By induction, \eqref{a-1} holds. Take $A=\lambda I$. Then

$$\partial_t(A^m)_i^j|_{t=0}=-m[\lambda^{m-1}(\text{\text{\text{Hess}}}  f) +f\lambda^{m+1}I+cf\lambda^{m-1}I]_i^j|_{t=0},$$
which is just \eqref{a-2}.

\end{proof}

\noindent Proposition \ref{a} implies the following
\begin{prop}
\begin{eqnarray}
\partial_t (\text{\r P}_r)_i^j|_{t=0} 
& = &\lambda^{r-1} \binom{n-2}{r-1} (\text{\r Hess} f)_i^j|_{t=0},\\
\label{var-pro2}
|\partial_t \text{\r P}_r|^2|_{t=0} & = &
\lambda^{2(r-1)} \binom{n-2}{r-1}^2 |\text{\r Hess} f|^2|_{t=0}, 
\end{eqnarray}
where $\text{\r Hess} f=\text{Hess}f-\frac{\Delta f}{n}I$.
\end{prop}

\begin{proof} By the definition \eqref{def-pr} of $P_r$,
        \begin{eqnarray}
\nonumber
&&
\partial_t (P_r)_i^j|_{t=0} \\
\nonumber& = & \partial_t \left( \sum_{m=0}^{r-1} (-1)^m s_{r-m} (A^m)_i^j \right )|_{t=0}\\
\nonumber
& = & \sum_{m=0}^{r-1} (-1)^m ( \partial_t s_{r-m}) (A^m)_i^j|_{t=0} + \sum_{m=0}^{r} (-1)^m s_{r-m} \partial_t (A^m)_i^j|_{t=0}\\
\nonumber
& = & \sum_{m=0}^{r-1} (-1)^m \left [ - \binom{n-1}{r-m-1}\lambda^{r-m-1}(\Delta f+nf\lambda^2+ncf)\lambda^m I \right ] \\
& +& \sum_{m=0}^{r} (-1)^m \binom{n}{r-m} \lambda^{r-m}[-m \lambda^{m-1}(\text{\text{\text{Hess}}} f  + f\lambda^2I+cf I)].
\end{eqnarray}
By  \eqref{combin-1} and \eqref{combin-4},
\begin{eqnarray}
\nonumber
&&
\partial_t (P_r)_i^j|_{t=0} \\
\nonumber
& = & - \binom{n-2}{r-1}\lambda^{r-1}(\Delta f+nf\lambda^2+ncf)I+\lambda^{r-1} \binom{n-2}{r-1} [\text{\text{Hess}} f +f\lambda^2 I+ fcI]\nonumber \\
& = & -\lambda^{r-1} \binom{n-2}{r-1}[(\Delta f) I + (n-1)f\lambda^2I +(n-1)cfI -\text{\text{Hess}} f ]_i^j.\label{pr-evol-0}
\end{eqnarray}

\noindent Note $\text{\r P}_r = P_r - \frac{n-r}{n} s_r
g$ and  $\partial_t g_i^j = 0$. By \eqref{var-sr1} and   \eqref{pr-evol-0}, we have

\begin{eqnarray}
 \nonumber
&&\partial_t (\text{\r P}_r)_i^j|_{t=0} \\
\nonumber
&=& \big[\partial_t (P_r)_i^j - \frac{n-r}{n} (\partial_t s_r) g_i^j - \frac{n-r}{n} s_r \partial_t g_i^j\big]|_{t=0}\\
\nonumber
&=& -\lambda^{r-1} \binom{n-2}{r-1}[(\Delta f) I + (n-1)f\lambda^2I+(n-1)cfI - \text{\text{Hess}} f ]_i^j\\
\nonumber
& &+ \frac{n-r}{n} \binom{n-1}{r-1}\lambda^{r-1}(\Delta f+nf\lambda^2+ncf)I_i^j\\
\nonumber
&=&\lambda^{r-1} \binom{n-2}{r-1} \left [\text{Hess} f - \frac{1}{n}(\Delta f)I +(n-1)cf I\right ]_i^j \\
& = &\lambda^{r-1} \binom{n-2}{r-1} (\text{\r Hess} f)_i^j.\label{var-pro}
\end{eqnarray}
\eqref{var-pro} yields 
\begin{eqnarray}
\nonumber
|\partial_t \text{\r P}_r|^2 |_{t=0} & = & \sum_{i,j=1}^{n}\partial_t (\text{\r P}_r)_i^j |_{t=0} \partial_t (\text{\r P}_r)_j^i |_{t=0}\\
& = &\lambda^{2(r-1)} \binom{n-2}{r-1}^2 |\text{\r Hess} f|^2\label{var-po}
\end{eqnarray}
\end{proof}

\noindent Take $\varphi=s_r$ in Prop \ref{average}.   It holds  that
      \begin{eqnarray}
\nonumber
\partial_t \overline{s_r}|_{t= 0} &=& \big[\overline{fH(s_r - \overline{s_r})}+ \overline{\partial_t s_r}\big]|_{t= 0}= \overline{\partial_t s_r} |_{t= 0} \\
\nonumber
&=& - \lambda^{r-1}\binom{n-1}{r-1} \overline{\Delta f +n\lambda^2f+ncf }|_{t= 0}\\
\nonumber
&=& - \lambda^{r-1}\binom{n-1}{r-1} (n\lambda^2+c)\overline{f}
\end{eqnarray}
We will choose $f$ later so that $\int_{\Sigma} f= 0$.
For such $f$,  
\begin{equation}\partial_t \overline{s_r}|_{t=0} = 0.\label{var-sra}
\end{equation}

\noindent Using \eqref{var-sr1}, \eqref{var-pro2} and \eqref{var-sra}, by Proposition \ref{r-secvar-2}, we may calculate the second variation of $\mathcal{G}$ at $t=0$ as follows:

\begin{eqnarray}
\nonumber
&&\frac{1}{2}\frac{d^2}{dt^2}\mathcal{G}(\Sigma_t)|_{t=0}\\
& =& C^2 \int_\Sigma |\partial_t \text{\r P}_r|^2 |_{t= 0} - \int_\Sigma [\partial_t (s_r-\overline{s_r})]^2|_{t= 0}\\
\nonumber
&=&  \lambda^{2(r-1)}\left[C^2\binom{n-2}{r-1}^2 \int_\Sigma |\text{\r Hess} f|^2 - \binom{n-1}{r-1}^2 \int_\Sigma (\Delta f + nf\lambda^2+ncf)^2\right]\\
\nonumber
&=&\lambda^{2(r-1)} C^2 \binom{n-2}{r-1}^2 \left [ \frac{n-1}{n} \int_\Sigma (\Delta f)^2  +\frac{n-1}{sn_c^2( a)}\int_\Sigma f\Delta f \right ]\\
\nonumber
& & -\lambda^{2(r-1)} \binom{n-1}{r-1}^2 \left [\int_\Sigma (\Delta f)^2 +2n(\lambda^2+c)\int_\Sigma f \Delta f + n^2(\lambda^2+c)^2 \int_\Sigma f^2 \right ].
\end{eqnarray}

By \eqref{combin-3}, $(n-1) \binom{n-2}{r-1} = (n-r) \binom{n-1}{r-1}$. So it holds  that

\begin{eqnarray}
\nonumber
&&\frac{1}{2}\frac{d^2}{dt^2}\mathcal{G}(\Sigma_t)|_{t=0}\\
\nonumber
&=&\lambda^{2(r-1)} \binom{n-1}{r-1}^2   \bigg[\left (\frac{C^2(n-r)^2}{n(n-1)}-1 \right ) \int_\Sigma (\Delta f)^2\\
\nonumber
& & + \left (C^2 \frac{(n-r)^2}{sn_c^2( a)(n-1)} - 2n(\lambda^2+c) \right )\int_\Sigma f \Delta f 
-n^2(\lambda^2+c)^2 \int_\Sigma f^2    \bigg]\\
&=&\lambda^{2(r-1)} \binom{n-1}{r-1}^2\left \{\alpha \int_\Sigma (\Delta f)^2 + \beta \int_\Sigma f \Delta f + \gamma \int_\Sigma f^2\right \}\label{var-G}
\end{eqnarray}
where $\alpha = \frac{C^2(n-r)^2}{n(n-1)}-1$, $\beta=C^2 \frac{(n-r)^2}{sn_c^2( a)(n-1)} - 2n(\lambda^2+c)$, $\gamma = -n^2(\lambda^2+c)^2.$

If $C<\sqrt{\frac{n(n-1)}{(n-r)^2}}$, then $\alpha<0$. Similar to the proof of Theorem \ref{thm-2der}, we can prove that

\begin{thm} \label{thm-2derr}Let $M_c^{n+1}$  be a space form of the constant sectional curvature $c$. Fix $2\leq r\leq n-1$. Given $ C<\sqrt{\frac{n(n-1)}{(n-r)^2}}$, there exists a closed  totally umbilical hypersurface $\Sigma$ in $M$ and a deformation $F(\cdot,t)$ of $\Sigma$ such that 
$$\frac{d^2}{dt^2}\mathcal{G}(\Sigma_t)|_{t=0}<0,$$ where 
$\mathcal{G}(\Sigma_t)$ is given by \eqref{f-2}.

\end{thm}

\noindent {\it The proof of Theorem \ref{opt-2}}. 
Let $F(x,t)$ be the deformation given in Theorem \ref{thm-2derr} with $\Sigma_0=\mathbb{S}^n$ and $C=C_1<\sqrt{\frac{n(n-1)}{(n-r)^2}}$. Then 
 Theorem \ref{thm-2derr} implies that for $t$ sufficiently small, 
 \begin{equation}\label{ine-opt5-1}\int_{\Sigma_t}\left(s_r-\overline{s_r}\right)^2> C_1^2\int_{\Sigma_t}|P_r-\frac{(n-r)s_r}{n}g|^2.
\end{equation} 
Take $(\Sigma_1)_t=\Sigma_t$ for sufficiently small $t$. So  \eqref{ine-opt2} holds. 

Given $C_2<\sqrt{n}$, let $C_1^2=\frac{n}{(n-r)^2}(C_2^2-1)$. Then $C_1<\sqrt{\frac{n(n-1)}{(n-r)^2}}$. By \eqref{ine-opt5-1} and
the identity:
\begin{equation}|P_r-\frac{(n-r)\overline{s_r}}{n}g|^2=|P_r-\frac{(n-r)s_r}{n}g|^2+\frac{(n-r)^2}{n}(s_r-\overline{s_r})^2,
\end{equation}
it holds that there exists a deformation $\Sigma_t$, denoted by $(\Sigma_2)_t$ so that
\begin{eqnarray}
\nonumber\int_{\Sigma_t}|P_r-\frac{(n-r)\overline{s_r}}{n}g|^2&>&\left(1+\frac{(n-r)^2C^2_1}{n}\right)\int_{\Sigma_t}|P_r-\frac{(n-r)s_r}{n}g|^2\\
\label{ine-opt5-2}
&=& C_2^2\int_{\Sigma_t}|P_r-\frac{(n-r)s_r}{n}g|^2.
\end{eqnarray}
Clearly, the Ricci curvature $\text{Ric}_{\Sigma_t}$ of $\Sigma_t$ is positive for $t$ sufficiently small.

\qed

\noindent  Xu Cheng\\Insitituto de Matem\'atica\\Universidade
Federal Fluminense - UFF\\Centro, Niter\'{o}i, RJ 24020-140 Brazil
\\e-mail:xcheng@impa.br

\bigskip
\noindent  Areli V\'azquez Ju\'arez\\Insitituto de Matem\'atica\\Universidade
Federal Fluminense - UFF\\Centro, Niter\'{o}i, RJ 24020-140 Brazil
\\e-mail: areli.vazquez@gmail.com

\bigskip


\begin{thebibliography}{9999}






\bibitem[Ca] {Ca} C.  Carathe\'odory, Calculus of variations and partial differential equations of the first order. Part I: Partial differential equations of the first order, Translated by Robert B. Dean and Julius J. Brandstatter, Holden-Day Inc., San Francisco, 1965.

\bibitem[C1]{C1} X.  Cheng,   A generalization of almost-Schur lemma for closed Riemannian manifolds, Ann  Global Anal and Geom,  43(2013),  153--160.


\bibitem[C2]{C2} X.  Cheng,  An almost-schur type lemma for symmetric (2,0) tensors and applications, arXiv:1208.2152v1.

\bibitem[CR]{CR} X. Cheng  and H. Rosenberg,  Embedded positive constant r-mean curvature hypersurfaces in $M^m \times \mathbb{R}$, Anais da Academia Brasileira de Ci\^encias (Annals of the Brazilian Academy of Sciences) (2005) 77(2), 183--199.


\bibitem[CZ]{CZ} X.  Cheng and D. Zhou, Rigidity for closed totally umbilical hypersurfaces in space forms, The Journal of Geometric Analysis, v. online, p. 1--10, 2012. 




\bibitem[dLM]{dLM} C. De Lellis and S. M$\ddot{\text{u}}$ller, Optimal rigidity estimates for nearly umbilical surfaces. J. Differential Geom. 69 (2005) 75--110.

\bibitem[dLT]{dLT} C. De Lellis and P. Topping, Almost -Schur Lemma,  Calc. Var. and PDE, 43 (2012) 347--354;  arXiv:1003.3527v2 [math.DG] 7 May 2011.

\bibitem[E]{E} L. C. Evans, Partial Differential Equations, Amer. Math. Soc., Providence, RI, 1993.


\bibitem[GW1]{GW1} Y. Ge and G. Wang,   An almost Schur theorem on $4$-dimensional manifolds, Proc. Amer. Math. Soc. 140 (2012), 1041--1044.

\bibitem[GW2]{GW2} Y. Ge and G. Wang,   A new conformal invariant on $3$-dimensional manifolds, arXiv:1103.3838, 2011.

\bibitem[GWX]{GWX} Y. Ge,  G. Wang and C. Xia,  On problems related to an inequality of Andrews, De Lellis and Topping, Int Math Res Notices (2012), doi: 10.1093/imrn/rns196.



\bibitem[H]{H} G. Huisken, Contracting convex hypersurfaces in Riemannian manifolds by their
mean curvature, Invent. Math. 84 (1986), 463--480.

\bibitem[HP]{HP} G. Huisken and A. Polden, Geometric evolution equations for hypersurfaces, Amer. Math. Soc. Transl. (2) 47 (1965), 89--129.

\bibitem[LM]{LM} T. Lamm and J. Metzger, Small surfaces of Willmore type in Riemannian man- ifolds, Int. Math. Res. Not. IMRN (2010), no. 19, 3786--3813. MR 2725514.

\bibitem[LM]{LMS} T. Lamm, J. Metzger, and F. Schulze, Foliations of asymptotically flat manifolds by surfaces of Willmore type, ArXiv e-prints (2009), arXiv:0903.1277v1 [math.DG]. 
\bibitem[M]{M} J. Metzger, Foliations of asymptotically flat 3-manifolds by 2-surfaces of prescribed mean curvature,  J. Differential Geom., 77(2):201--236, 2007.

 


\bibitem[P]{P} D. Perez, On nearly umbilical hypersurfaces, thesis, 2011.

\bibitem[R]{R} R. C. Reilly, Variational properties of functions of the mean curvatures for hypersur- faces in space forms, J. Differential Geometry, 8 (1973), 465--477.

\bibitem[Ro]{Ro} H. Rosenberg, Hypersurfaces of constant curvature in space forms, Bull. Sci. Math., 117 (1993), 211--239.


\bibitem[Ro1]{Ro1}A.  Ros, Compact hypersurfaces with constant scalar curvature and a congruence theorem. J. Differential Geom. 27 (1988), 215--220.
\bibitem[Ro2]{Ro2} A. Ros,  Compact hypersurfaces with constant higher order mean curva tures. Rev. Mat. Iberoamericana 3 (1987), 447--453.



\end{thebibliography}
\end{document}